\documentclass[12pt]{article}

\usepackage{amsthm,amsfonts, amsbsy, amssymb,amsmath,graphicx}
\usepackage{graphics}
\usepackage[english]{babel}
\usepackage{graphicx}
\usepackage{lineno}
\usepackage{enumerate}
\usepackage{tikz}
\usepackage{tkz-graph}
\usepackage{tkz-berge}
\usepackage{hyperref}
\usepackage{color}
\usepackage{tikz}
\usepackage{float}
\hypersetup{colorlinks=true}
\usepackage{todonotes}
\hypersetup{colorlinks=true, linkcolor=blue, citecolor=blue,urlcolor=blue}
\usepackage{algorithm}
\usepackage{algpseudocode}
\algnotext{EndFor}
\algnotext{EndIf}

\newtheorem{theorem}{Theorem}[section]

\newtheorem{lemma}[theorem]{Lemma}
\newtheorem{proposition}[theorem]{Proposition}

\newtheorem{corollary}[theorem]{Corollary}
\newtheorem{remark}[theorem]{Remark}
\newtheorem{fact}[theorem]{Fact}

\newcommand{\mmd}{{\rm mmd}}
\newcommand{\ex}{{\rm ex}}
\newcommand{\twin}{{\rm t}}
\newcommand{\res }{{\rm res}}
\newcommand{\td}{{\rm td}}

\newcommand{\move}{\rightsquigarrow}

\tikzstyle{new style 0}=[fill={rgb,255: red,191; green,191; blue,191}, draw=none, shape=circle, tikzit shape=circle]
\tikzstyle{new style 1}=[fill=white, draw=black, shape=circle]
\tikzstyle{new style 2}=[fill=none, draw=none, shape=circle]

\tikzstyle{new edge style 0}=[-]

\begin{document}
\title{Completely mobile robots in graphs with a resolving property}

\author{Jos\'e C. Camacho Moreno $^{a}$ \\ \texttt{\footnotesize josecarlos.camacho@uca.es}  \\
\and
  Dorota Kuziak $^{b}$ \\ \texttt{\footnotesize dorota.kuziak@uca.es} \\
\and
  Ethan Shallcross $^{c}$ \\ \texttt{\footnotesize ershallcross@gmail.com} \\
\and
  James Tuite $^{c,d,e}$ \\ \texttt{\footnotesize james.t.tuite@open.ac.uk} \\
\and
  Ismael G. Yero $^{a}$ \\ \texttt{\footnotesize ismael.gonzalez@uca.es} \\
}

\maketitle

\begin{center}
$^a$ Departamento de Matem\'aticas, Universidad de C\'adiz, Algeciras Campus, Spain\\[0.1cm]

$^b$ Departamento de Estad\'istica e IO, Universidad de C\'adiz, Algeciras Campus, Spain\\[0.1cm]

$^c$ School of Mathematics and Statistics, Open University, Milton Keynes, UK\\[0.1cm]

$^d$ Department of Informatics and Statistics, Klaip\.{e}da University, Lithuania

$^e$ Faculty of Mathematics and Informatics, Vilnius University, Lithuania
\end{center}

\begin{abstract}
This work extends the study of mobile resolvability in connected graphs by introducing completely mobile resolving sets. As with mobile resolving sets, we consider stationing a collection of robots at the vertices in a resolving set of a graph. Robots may move one at a time to any adjacent vertex as long as set of occupied vertices remains a resolving set. In this new variant, we ask that every robot can traverse every vertex in the graph through a sequence of permitted moves.

Graphs possessing at least one completely mobile resolving set are termed mobile-resolvable. We completely characterise the non-mobile-resolvable graphs as paths of order at least three and subdivided stars with three or more leaves and at most two subdivided edges. For trees, we explicitly determine the completely mobile metric dimension, showing that it exceeds the classical metric dimension by one for most mobile-resolvable trees, and by two for certain subdivided stars. Furthermore, we evaluate this parameter across fundamental graph classes—including cycles, complete bipartite graphs, grid graphs, and circulant graphs.
\end{abstract}


\noindent
{\bf AMS Subj.\ Class.\ (2020)}: 05C12; 05C76

\section{Introduction}\label{sec:intro}

The concept of distance-based resolvability in graphs, introduced independently in the 1970s by Slater~\cite{Slater} as well as by Harary and Melter~\cite{Harary}, has long served as a fundamental framework for uniquely identifying vertices within a network structure. Given a connected graph $G$, a subset of vertices $S \subseteq V(G)$ is called a \emph{resolving set} for $G$ if every vertex of $G$ is uniquely determined by its vector of distances to the vertices in $S$. The minimum cardinality among all such subsets is the \emph{metric dimension} of $G$, denoted by $\dim(G)$.

Originally motivated by problems in sonar/radar tracking and landmark-based navigation, the classical metric dimension and its variations have found widespread applications across diverse domains, including wireless sensor network positioning, chemical structure identification, robotic navigation, privacy in social networks, and information propagation analysis in complex networks. On the other hand, to address real-world operational constraints and varying network topologies, researchers have developed numerous variations of classical resolvability. Some examples include fault-tolerant metric dimension, local metric dimension, partition dimension, fractional metric dimension, and multiset dimension, each modifying the structural conditions placed on the resolving subset or the distance vectors. In order to avoid including many different citations for these applications and variations, we instead direct the reader to the four recent surveys~\cite{Albejani,Farhan,dorota-2022+,till-2022+}, which cover almost all investigations on resolvability parameters of graphs. In addition, we mention the papers~\cite{Cameron,Farhan26,Foster,Knor,Korivand}, which are some of the most recent articles on the classical metric dimension of graphs.

Traditional models of resolvability in graphs have assumed a static deployment of landmarks or sensors. In~\cite{KlaKriTuiYer-2023} a dynamic version of another popular graph theory topic, the general position problem, was introduced, and this was later extended to mutual visibility sets in~\cite{Detlaff} and resolving sets in~\cite{camacho}. In~\cite{camacho}, a small number of autonomous agents (which we picture as mobile robots) traverse the graph over discrete time steps, so that each node is at some point visited by a robot. At each step, the positions of the robots must collectively form a resolving set.

While the original mobile resolving set model guarantees continuous network monitoring with a minimal number of mobile agents, many practical scenarios might demand stricter operational criteria. For example, in multi-robot exploration, network security patrolling, infrastructure inspection, or environmental mapping, every individual agent must actively inspect and traverse every node of the network. In other words, we require not only that every node of the network can be inspected by some robot, but that every robot is able to visit every node. Such a problem was considered in the context of the general position problem in~\cite{Shallcross-2025}, under the name \emph{completely mobile general position sets}. In this paper we study \emph{completely mobile resolving sets}. In this model, the configuration of robots must always form a resolving set, and every robot must be able to visit every vertex of the graph.

We now formalise all the key terminology and notation that we need in our exposition. We consider a collection of robots positioned at a set of nodes of a network, with at most one robot per vertex, such that the occupied vertices form a resolving set; that is, each vertex of the graph is uniquely identified by its distances to the robots. At each step, one robot can move along an edge to an unoccupied vertex if the resulting configuration is also a resolving set. Through a sequence of such moves, we require that each robot can visit every vertex of the graph. To optimise resources, we ask for the smallest possible number of robots that can fulfil these conditions.

We assume that $G$ is a connected, undirected and simple graph without loops or multiple edges. The set of vertices and edges of $G$ shall be written as $V(G)$ and $E(G)$, respectively. Given two vertices $u,v\in V(G)$, the \emph{distance} between $u$ and $v$ is the length of a shortest $u,v$-path, and is denoted by $d_G(u,v)$ (or simply $d(u,v)$ if $G$ is clear from the context). The two vertices $u,v$ are \emph{resolved}, \emph{identified} or \emph{recognised} by a vertex $x\in V(G)$ if $d_G(u,x)\ne d_G(v,x)$; we also say that $x$ resolves, identifies or recognises the pair of vertices $u,v$. Moreover, by $R\{x, y\}$ we mean the set of all vertices of $G$ that resolve the pair $x,y$ (note that $x,y$ are themselves in $R\{x, y\}$). In this sense, a set of vertices $S\subseteq V(G)$ is a \emph{resolving set} of $G$ if every pair of vertices of $G$ is resolved by some vertex of $S$. A resolving set with the smallest possible cardinality is called a \emph{metric basis}. The cardinality of a metric basis in $G$ is the \emph{metric dimension} of $G$, denoted by $\dim(G)$. Moreover, if $S=\{v_1,\dots,v_k\}$ is any subset of $V(G)$, then for any vertex $x\in V(G)$, the \textit{metric representation} of $x$ is the vector $r(x|S)=(d_G(x,v_1),\dots,d_G(x,v_k))$.

Given a resolving set $S$ of a graph $G$ occupied by a set of robots, and a vertex $v$ of $S$, the robot located at the vertex $v$ makes a \emph{valid} or \emph{legal resolving movement}, if it moves from $v$ to an unoccupied neighbour $v'$, such that the new set $(S\setminus \{v\} )\cup \{v'\}$ is also a resolving set of $G$. A valid or legal resolving movement from $u$ to $v$ shall be denoted by $u\move v$. In addition, if we set $S' = (S\setminus \{v\} )\cup \{v'\}$, then we can also represent the move $u\move v$ as a set version as $S\move S'$.

Having these notions in mind, a given set of vertices $S\subseteq V(G)$ is called a \emph{mobile resolving set} for the graph $G$ if each vertex of $G$ can be visited by at least one robot from the set of robots initially placed at vertices of $S$ by some sequence of valid resolving movements. The \emph{mobile metric dimension} of $G$ is the cardinality of a smallest mobile resolving set of $G$ and is denoted by $\mmd(G)$. A mobile resolving set of cardinality $\mmd(G)$ is called a \emph{mobile metric basis}. The concepts above were first introduced and studied in~\cite{camacho}.

We now consider strengthening our constraints. We say that the set $S\subseteq V(G)$ is a \emph{completely mobile resolving set} if each vertex of $G$ can be visited by each of the robots initially placed at vertices of $S$ by a sequence of valid resolving movements. The \emph{completely mobile metric dimension} of $G$ is the cardinality of a smallest completely mobile resolving set of $G$ and is denoted by $\mmd_c(G)$. A completely mobile resolving set of cardinality $\mmd_c(G)$ is called a \emph{completely mobile metric basis}. From now on, in order to facilitate our exposition, by `a robot $x$' we mean a robot that is placed at a vertex $x\in V(G)$.

One might immediately notice that not every graph $G$ contains a completely mobile resolving set. A trivial example is a path graph $P_n$ on at least three vertices. It is a well-known result that $P_n$ has metric dimension one and only two metric bases, each one formed by any of its two leaves. Clearly, a robot placed at the unique vertex of a metric basis of a $P_n$ ($n\ge 3$) cannot make any valid resolving movement, and so a completely mobile metric basis of $P_n$ must have at least two robots $x$ and $y$; however, the robot $y$ is unable to cross the robot $x$ in order to visit the leaf of $P_n$ closest to $x$, so that such a configuration cannot be completely mobile. Thus $P_n$ has no completely mobile resolving set for $n \geq 3$. From now on, we say that $G$ is \emph{mobile-resolvable} if it contains at least one completely mobile resolving set.

\subsection{Basic terminology and required results}\label{subsec:basic}

Throughout this paper, we shall write $[k]=\{1,\dots,k\}$ and $[k]^*=\{0,1,\dots,k\}$. The \textit{Cartesian product} of two graphs $G$ and $H$ is the graph $G\Box H$ whose vertex set is $V(G\Box H)=V(G)\times V(H)$, and two vertices $(g,h), (g',h')$ are adjacent in $G\Box H$ if either $g=g'$ and $hh'\in E(H)$; or $gg'\in E(G)$ and $h=h'$. The vertex set of the cycle graph $C_n$ will be identified with the cyclic group $\mathbb{Z}_n$ in the natural way, so all operations are performed modulo $n$. A graph $G$ is \emph{$k$-connected} if for any subset $S \subseteq V(G)$ with $|S| < k$ the graph $G-S$ is connected. If there is a vertex $u$ such that $G-u$ is disconnected, then $u$ is a \emph{cut vertex}. A graph containing a cut vertex is \emph{separable}.

A vertex of degree at least~$3$ in a tree $T$ is called a \emph{major vertex}. A leaf $u$ of $T$ is called a \emph{terminal leaf} of a major vertex $v$ of $T$ if $d_T(u,v) < d_T(u,w)$ for every other major vertex $w$ of $T$. The \emph{terminal degree} $\td(v)$ of a major vertex $v$ is the number of terminal leaves associated with~$v$. A major vertex $v$ of $T$ is an \emph{exterior major vertex} of $T$ if $\td(v)\ge 1$. The path joining a terminal leaf with its corresponding major vertex without the major vertex itself is called a \emph{terminal path}. Clearly, every vertex on a terminal path in $T$ is either a leaf of $T$ or has degree~$2$ in $T$. A vertex on a terminal path that has degree~$2$ in $T$ is called an \emph{internal terminal vertex}. If a terminal path has at least one internal terminal vertex, then it is called a {\em non-trivial terminal path}. In this sense, a trivial terminal path is formed by a single vertex, which is a leaf of $T$ and adjacent to the exterior major vertex. We represent the number of leaves of $T$ by $n_1(T)$, and by $\ex(T)$ we mean the number of exterior major vertices of $T$.

The value of the classical metric dimension of trees has been reported in different works (see the survey~\cite{till-2022+}). For any tree $T$,
\begin{equation}\label{eq-dim-trees}
\dim(T)=n_1(T)-\ex(T).
\end{equation}

In addition to the formula above, it is known that the structure of any metric basis of a tree is as follows. For each exterior major vertex $v$ of $T$ with $\td(v)=k\ge 1$, we take $k-1$ vertices, in such a way that all but one of the terminal paths of $v$ contain exactly one vertex from the metric basis. Figure \ref{fig-tree-basis} contains one example of a metric basis of a tree.

\begin{figure}[h]
\centering
\begin{tikzpicture}[scale=.57, transform shape]
\node [draw, shape=circle] (x1) at  (0,0) {};
\node [draw, shape=circle] (y1) at  (2,0) {};
\node [draw, shape=circle, fill=black] (z1) at  (4,0) {};
\node [draw, shape=circle] (x2) at  (7,0) {};
\node [draw, shape=circle] (y2) at  (9,0) {};
\node [draw, shape=circle] (z2) at  (11,0) {};
\node [draw, shape=circle] (x3) at  (14,0) {};
\node [draw, shape=circle] (y3) at  (16,0) {};
\node [draw, shape=circle] (z3) at  (18,0) {};
\node [draw, shape=circle] (x4) at  (21,0) {};
\node [draw, shape=circle] (y4) at  (23,0) {};
\node [draw, shape=circle] (z4) at  (25,0) {};
\node [draw, shape=circle, fill=black] (y22) at  (9,-2.5) {};
\node [draw, shape=circle, fill=black] (x33) at  (14,-2.5) {};
\node [draw, shape=circle, fill=black] (z33) at  (18,-2.5) {};
\node [draw, shape=circle, fill=black] (x44) at  (21,-2.5) {};
\node [draw, shape=circle, fill=black] (y44) at  (23,-2.5) {};
\node [draw, shape=circle] (z44) at  (25,-2.5) {};
\node [draw, shape=circle, fill=black] (x444) at  (0,-2.5) {};
\node [draw, shape=circle] (y444) at  (4,-2.5) {};
\node [draw, shape=circle, fill=black] (z444) at  (7,-2.5) {};
\node [draw, shape=circle, fill=gray] (c1) at  (2,3) {};
\node [draw, shape=circle, fill=gray] (c2) at  (9,3) {};
\node [draw, shape=circle] (x) at  (12.5,3) {};
\node [draw, shape=circle, fill=gray] (c3) at  (16,3) {};
\node [draw, shape=circle, fill=gray] (c4) at  (23,3) {};
\node [draw, shape=circle, fill=gray] (d1) at  (5.5,3) {};
\node [draw, shape=circle, fill=gray] (d2) at  (19.5,3) {};
\node [draw, shape=circle] (e1) at  (5.5,5) {};
\node [draw, shape=circle] (e2) at  (19.5,5) {};

\draw(x1)--(c1)--(y1);
\draw(c1)--(z1);
\draw(d1)--(e1);
\draw(d2)--(e2);

\draw(x444)--(x1);
\draw(y444)--(z1);
\draw(z444)--(x2);

\draw(c1)--(d1)--(c2)--(x)--(c3)--(d2)--(c4);
\draw(y22)--(y2)--(c2)--(x2);
\draw(c2)--(z2);

\draw(x33)--(x3)--(c3)--(z3)--(z33);
\draw(c3)--(y3);

\draw(x44)--(x4)--(c4)--(z4)--(z44);
\draw(c4)--(y4)--(y44);
\end{tikzpicture}
\caption{A tree $T$ provided with a metric basis in bold. Exterior major vertices appear in grey.}
\label{fig-tree-basis}
\end{figure}
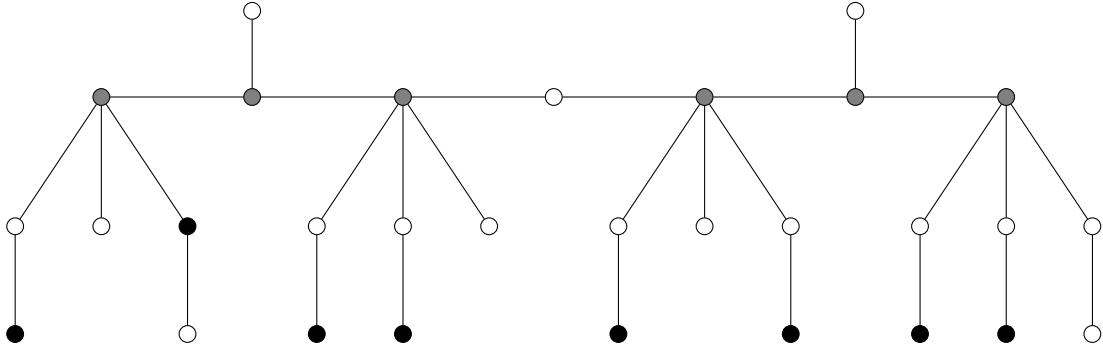

A particular family of trees that is of interest in our exposition is formed by subdivided star graphs. That is, if $S_n$ is the star graph on $n$ leaves (whose order is $n+1$) and $k\ge 0$, then we say that $S_n(k)$ is a \textit{subdivided star} if it is obtained from a star $S_n$ by subdividing $k$ of its edges. By subdividing an edge $uv$ we mean removing the edge $uv$, adding a path $P_\ell=12\cdots \ell$ (with $\ell\ge 1$), and adding the edges $u1,\ell v$. Notice that $S_n(0)$ is simply the star $S_n$ and that $S_n(k)$ has order at least $n+k+1$. We remark that, in our work, we shall not take into account how many vertices are used to subdivide an edge, only that at least one is used.

\subsection{Plan of the article}

The rest of the article is organised into four main sections. In Section~\ref{sec:basic}, we establish fundamental bounds for the completely mobile metric dimension, investigate graphs which attain these bounds, and show that the difference between the mobile and completely mobile metric dimensions of a graph can be arbitrarily large. We give the completely mobile metric dimensions of basic graph families such as cycles, grids, and complete graphs. Section~\ref{sec:exclud-non-res} delivers the primary theoretical results of the paper, by completely characterising non-mobile-resolvable graphs (paths of order at least three and specific subdivided stars) and determining the exact completely mobile metric dimension for all mobile-resolvable trees. Section~\ref{sec:circulant} focuses on circulant graphs $C_n(2)$, providing exact dimensional values and explicit robot movement patterns based on $n \pmod 4$. Finally, Section~\ref{sec:open} concludes by giving several potential directions for future study.

\section{A few basic results}\label{sec:basic}

Clearly, any completely mobile resolving set is a mobile resolving set, which is also a resolving set. Trivially, the whole vertex set of a graph $G$ of order $n \geq 2$ cannot be a completely mobile resolving set of $G$, as there are no available valid moves. Hence for any mobile-resolvable graph $G$ of order $n \geq 2$,

\begin{equation}\label{eq-completely-mobile-not-tot}
1\le \dim(G)\le \mmd(G)\le \mmd_c(G)\le n-1.
\end{equation}

We first study completely mobile resolving sets in some basic graph classes. As previously discussed, every path $P_n$ is not mobile-resolvable, unless $n \leq 2$, in which case it is clear that $\mmd_c(P_2)=1$. The path graph is the only graph with metric dimension one, so for any other mobile-resolvable graph $G$ we have $\mmd_c(G) \geq 2$. It is also easily seen that for any complete graph $K_n$ with $n \geq 2$ we have $\mmd_c(K_n)=n-1$.

We now determine the completely mobile metric dimension of the cycle graph $C_n$ (recall that we take the vertex set of $C_n$ to be $\mathbb{Z}_n$). For $n \ge 5$ it is readily observed that for any $k \in \mathbb{Z}_n$ the pairs $\{ k,k+1\} $ and $\{ k, k+2\} $ are resolving sets of $C_n$. Since the set $\{k, k+2\}$ can be obtained from $\{k, k+1\}$ by performing the move $k+1\move k+2$, we can start with the resolving set $\{ 0,1\} $ and repeat the sequence $\{ k,k+1\} \move \{ k,k+2\} \move \{ k+1,k+2\} $ for $k = 0,1,\ldots ,n-2$, and in this manner both robots visit every vertex of the cycle by valid moves. Hence $\mmd_c(C_n) = 2$ for $n \geq 5$. For the smaller cases, we have $\mmd_c(C_3)=2$ and $\mmd_c(C_4)=3$.

The example of cycles of order at least five shows that equality can hold for the three parameters $\dim(G)$, $\mmd(G)$ and $\mmd_c(G)$ in Inequality~\ref{eq-completely-mobile-not-tot}. That is, for any cycle $C_n$ with $n\ge 5$,

\begin{equation}
    \label{eq:cycles}
    \dim(C_n) = \mmd(C_n) = \mmd_c(C_n)=2.
\end{equation}

Another classic family of graphs which give equality in part of the chain from~\eqref{eq-completely-mobile-not-tot} are the so-called grid graphs (the Cartesian product of two paths $P_n$ and $P_m$). We make the identification $V(P_r)=[r]$ for $r \geq 1$. Let $(1,1),(1,m),(n,1),(n,m)$ be the corners of the grid $P_n\Box P_m$. It is known that any two `consecutive' corners (for instance $(1,1)$ and $(1,m)$) of $P_n\Box P_m$ form a metric basis for $P_n\Box P_m$. It was proven in~\cite{camacho} that $\mmd(P_n\Box P_m)=3$ for any integers $n,m\ge 2$. Thus, from \eqref{eq-completely-mobile-not-tot}, we have that $\mmd_c(P_n\Box P_m)\ge 3$. When two robots are stationed at consecutive corners, the third robot can move freely through the rest of the grid, and also three robots can move around the border of the grid so that there are always two robots at consecutive corners. From these observations, it is easily seen that any initial configuration of three robots with two robots at consecutive corners is a completely mobile resolving set.

\begin{remark}
For any integers $n,m\ge 2$, $\mmd(P_n\Box P_m)=\mmd_c(P_n\Box P_m)=3$.
\end{remark}

For 2-connected graphs $G$, we can give an elementary bound on the completely mobile metric dimension in terms of the \emph{resolving number} $\res (G)$ introduced in~\cite{resolving}, which is defined to be the smallest value of $k$ such that every $k$-subset of $G$ is a resolving set. For any graph we have $\res (G) \leq n-1$.

\begin{lemma}\label{resolving bound}
    If $G$ is 2-connected, then $\mmd _c(G) \leq \res (G)$. In particular, $G$ is mobile-resolvable.
\end{lemma}
\begin{proof}
    Let $G$ be 2-connected and set $k = \res (G)$. Start with any configuration of $k$ robots on $G$; we wish to show that this configuration is a completely mobile resolving set. Note that any move of a robot to an unoccupied neighbour is a valid move. Suppose for a contradiction that there is a pair of vertices $u,v \in V(G)$ such that there is a robot $R$ at $u$ and there is no sequence of valid moves that takes $R$ to $v$. Choose the pair $u,v$ so that $d(u,v)$ is as small as possible. Let $u'$ be a neighbour of $u$ with $d(u',v) = d(u,v)-1$. Let $P$ be shortest path in $G-u$ from $u'$ to the set of unoccupied vertices. By moving robots along these paths, we can make $u'$ unoccupied, and then move $R$ to $u'$; however, there can be no series of valid moves taking $R$ from $u'$ to $v$, contradicting the minimality of $d(u,v)$.
\end{proof}

We next show that there are also several graph classes in which the equality mentioned above never happens. In fact, the difference $\mmd_c(G) - \mmd(G)$ can be arbitrarily large.

\begin{theorem}
For all $q \geq 1$, there exists a graph $G$ with $\mmd_c(G) - \mmd(G) \geq q$.
\end{theorem}
\begin{proof}
Let $k \geq 3$ be an integer. As in~\cite{camacho}, we define $G_k$ to be the graph of order $2^k + k$ with $V(G) = A \cup B$, where $A = \mathbb{Z}_k$ and $B = \{ S : S \subseteq A \}$. The edges of $G_k$ are defined as follows: (i) both subsets $A$ and $B$ induce a clique in $G_k$; and (ii) a vertex $i \in A$ is adjacent to every vertex $A_j\in B$ such that $i\in A_j$. It was shown in~\cite{camacho} that $\mmd(G_k) = k+1$.

We now show that $\mmd_c(G_k) = 2^{k-1} + k-1$. For $0 \leq a < b < k$, let $B_{a,b}$ be the set of subsets of $\{ 0, \dots, k-1 \}$ which contain, either both $a$ and $b$, or neither $a$ nor $b$ (so $|B_{a,b}| = 2^{k-1}$). If $T_1, T_2 \in B \setminus B_{a,b}$ are distinct sets, then either $T_1 \setminus \{ a,b \} \neq T_2 \setminus \{ a,b \}$ or $\{ a, b \} \subseteq (T_1 \setminus T_2) \cup  (T_2 \setminus T_1)$. Therefore, both $(A \setminus \{ a \}) \cup B_{a,b}$ and $(A \setminus \{ b \}) \cup B_{a,b}$ are resolving sets of $G_k$. We note also that any $S \subseteq V(G_k)$ with $A \subseteq S$ is a resolving set, and so, whenever there is a robot at each vertex of $A$, the robots in $B$ can move freely.

Begin by stationing robots at $(A \setminus \{ 0 \}) \cup B_{0,1}$. First we show that for any robot $R$ in the configuration there is a sequence of valid moves that brings $R$ to the position $\{ 0,1\} \in B$ such that the set of occupied vertices remains $(A \setminus \{ 0 \}) \cup B_{0,1}$. If $R$ is already the robot at $\{ 0,1\} $, then we are done. If $R$ is in $B \setminus \{ \{ 0,1\} \} $, then perform the move $\{ 0,1\} \move 0$ and then move the robots in $B$ to the configuration $B_{0,1} \setminus \{\{ 0,1,2 \}\}$, with $R$ occupying $\{ 0,1\} $. Then, send the first robot to $\{ 0,1,2 \} $ via the move $0 \move \{ 0,1,2\} $.

If $R$ is at the vertex $i \in A$, move the robot at $\{ 0,1\} $ to $0$ and then move the robots remaining in $B$ to the configuration $B_{0,i} \setminus \{ \{ 0,i\} \} $. Next, send the robot $R$ to $B$ via the move $i \move \{ 0,i\} $ and make the move $\{ 0,i,j\} \move i$, where $j \in A \setminus \{ 0,1\} $. Finally, move the robots in $B$ to the configuration $B_{0,1} \setminus \{ \{ 0,1,2\} \}$, with $R$ occupying $\{ 0,1\} $ and bring the first robot back to $B$ via the move $0 \move \{ 0,1,2\} $.

Now that we have shown that any robot can be made to occupy the vertex $\{ 0,1\} $ in $(A \setminus \{ 0 \}) \cup B_{0,1}$, it suffices to prove that the robot $R'$ at $\{ 0,1\} $ can visit any vertex of $G_k$ via a series of valid moves. Make the move $\{ 0,1,2\} \move 0$, then $R'$ can visit any vertex of $B$. Finally, to allow $R'$ to visit $j \in A$, move the remaining robots in $B$ to the configuration $B_{i,j_1} \setminus \{ \{ i,j_1,j_2\} \}$, where $i,j_1,j_2$ are distinct elements of $A$, so that $R'$ is stationed at $\{ i,j_1\} $, then make the two moves $i \move \{ i,j_1,j_2\}$ and $\{ i,j_1\} \move i$.





For the lower bound, let $S$ be a completely mobile resolving set of $G_k$, and suppose for a contradiction that $|S| < 2^{k-1} + k-1$. Since each robot must traverse each vertex in $A$, there is a step at which a robot moves from $A$ to $B$, and so we may assume that $|S \cap A| < k$. As reasoned in the proof of~\cite[Lemma 2.9(ii)]{camacho}, it follows that $|S \cap B| \geq 2^{k-1}$. Therefore, $|S \cap A| \leq k-2$.

Let $a,b \in A \setminus S$ be distinct vertices. For each $T \in B$ such that $T \cap \{ a,b \} = \emptyset$, the vertices $T$, $T \cup \{ a \}$, $T \cup \{ b \}$, and $T \cup \{ a,b \}$ are equidistant from every other vertex in $V(G_k) \setminus \{ a, b \}$. It follows that at least three of these four vertices must be occupied. There are $2^{k-2}$ such sets $T$. Therefore, $|S \cap B| \geq 3 \times 2^{k-2} \geq 2^{k-1} + k-1$ --- a contradiction.
\end{proof}

We now discuss graphs which have very small or large completely mobile metric dimension. It is trivial that the only graphs with completely mobile metric dimension one are $K_1$ and $K_2$, since the only graphs with metric dimension one are the paths $P_n$, which we have already established are not mobile-resolvable for $n \geq 3$. We need one lemma that gives a lower bound for the mobile metric dimension in terms of the maximum degree.

\begin{lemma}\label{lem:Delta lower bound}
    For any graph with order $n \geq 2$ and maximum degree $\Delta $, the completely mobile metric dimension satisfies \[ \mmd_c (G) \geq \mmd (G) \geq 1+\left \lceil \log _3(\Delta ) \right \rceil .\] If $G$ is bipartite, this can be strengthened to \[ \mmd_c (G) \geq \mmd (G) \geq 1+\left \lceil \log _2(\Delta ) \right \rceil .\]
\end{lemma}
\begin{proof}
    Consider any mobile metric resolving set $S$ of $G$, and let $x$ be a vertex of $G$ of maximum degree. At some point there is a robot at $x$. Let $R$ be any robot at a vertex $y \neq x$ of $G$. If $t = d(y,N(x)) = \min \{ d(y,z): z \sim x\} $, then for any $z \in N(x)$ we have $d(y,z) \in \{ t,t+1,t+2\} $. Therefore by the pigeonhole principle at least $\frac{\Delta }{3}$ vertices of $N(x)$ have the same distance to robot $R$. By induction, if there are $k+1$ robots in $S$, then there is a subset $A$ of $\frac{\Delta }{3^k}$ neighbours of $x$ such that each robot is equidistant to every vertex of $A$. As $S$ is a resolving set, we require that $\frac{\Delta }{3^k} \leq 1$. It follows that there are at least $1+\log _3(\Delta )$ robots in $S$. If $G$ is bipartite, $R$ is a robot at $y \neq x$ and $t = d(y,N(x))$, then $y$ can be at distance $t$ or $t+2$ from any neighbour of $x$, and the same argument follows with $2^k$ in place of $3^k$.
\end{proof}

Characterising the graphs $G$ with $\mmd _c(G) = 2$ seems to be a hard problem. We have already seen that the cycles $C_n$, $n \geq 5$, are a family with completely mobile metric dimension two, and it follows from Lemma~\ref{lem:Delta lower bound} that any other graph $G$ with $\mmd _c(G) = 2$ must have maximum degree three. We now describe two such families.

\begin{theorem}
 If a graph $G$ is obtained by adding a chord to an odd cycle, then $\mmd _c(G) = 2$.
\end{theorem}
\begin{proof}
    As before, label the vertices of the cycle with $\mathbb{Z}_n$, where $n$ is odd. Relabeling the vertices of $C_n$ if necessary, we may assume that the added chord is the edge $0r$, where $r$ is odd. Set $m = \frac{n+r}{2}$. Define the map $\sigma (i) = r-i$ for all $i \in \mathbb{Z}_n$. The map $\sigma $ is an automorphism of $G$ that swaps $0$ and $r$, and has $m$ as unique fixed point. Observe that $d(i,m) = d(j,m)$ if and only if $j = \sigma (i)$. Using this fact, it is easily verified that $\{ w,m\} $ is a resolving set for any $w \neq m$, and that $\{ m-1,m+1\} $ is also a resolving set. Start with a robot $R_1$ at $m$ and a second robot $R_2$ can move freely through the remaining vertices. The two robots can also swap positions by moving $R_2$ to $m-1$, then moving $R_1$ to $m+1$, and then $R_2$ to $m$. At this point $R_2$ has visited every vertex, and now $R_1$ can now visit the remaining vertices while $R_2$ is stationed at $m$.
\end{proof}

\begin{theorem}
If $G$ is a unicyclic graph formed by adding at most one path to each vertex of an odd cycle, then $\mmd _c(G) = 2$.
\end{theorem}
\begin{proof}
    Let the vertex set of the odd cycle be $u_i$, $i \in \mathbb{Z}_n$, where $n = 2k+1$ for some integer $k \geq 1$. Suppose that for $i \in \mathbb{Z}_n$ the vertex $u_i$ has a path of length $\ell (i) \geq 0$ attached to it. Then we can denote the vertices by $u_{ij}$, where $0 \leq j \leq \ell (i)$ and $u_{i0} = u_i$. We claim that the set $S = \{ u_{rs},u'\} $ is a resolving set, where $0 \leq s \leq \ell (r)$ and  $u'$ is a vertex antipodal to $u_{r0}$ on the cycle, i.e. $u'$ is one of $u_{r+k},u_{r-k}$. Without loss of generality, we take $S = \{ u_{0s},u_k\} $, $0 \leq s \leq \ell (0)$.

    Firstly, it is easily verified that the vertices $u_{0j}$ are uniquely identified by $S$. For $0 \leq j \leq s$ there are just two vertices at distance $j$ from $u_{0j}$, which have different distances to $u_k$. If $j > 2s$, then the distance of $u_{0j}$ is $j-s$, and any other vertex of $G$ at distance $j-s$ from $u_{0s}$ is at distance $j-2s$ from $u_0$, and so is at distance at most $j-2s-1$ from $u_{1}$ or $u_{n-1}$. In either case, such a vertex is at distance at most $j-2s-1+k < j+k = d(u_{0j},u_{k})$ from $u_k$. Therefore, suppose that there is a pair of vertices $u_{i_1j_1},u_{i_2,j_2}$, $i_1,i_2 \neq 0$, that is not resolved by $\{ u_{0s},u_{k}\} $. Suppose that $i_1,i_2 \in [1,k] $; then the distance vectors to $S$ are $(s+i_1+j_1,k-i_1+j_1)$ and $(s+i_2+j_2,k-i_2+j_2)$. Setting these vectors to be equal and solving the resulting equations shows that $i_1 = i_2$ and $j_1 = j_2$. Similarly if $i_1,i_2 \in [k,n-1]$, the distance vectors are $(j_1+n-i_1+s,j_1+i_1-k)$ and $(j_2+n-i_2+s,j_2+i_2-k)$ and we again conclude that $i_1 = i_2,j_1 = j_2$. Finally suppose that $i_1 \in [1,k], i_2 \in [k+1,n-1]$. Then the two distance vectors are $(s+i_1+j_1,k-i_1+j_1)$ and $(j_2+n-i_2+s,j_2+i_2-k)$. Hence we can derive the equation $2j_1+k=n+2j-k$, which is a contradiction, since $n$ is odd.

    The two robots can traverse the cycle whilst always maintaining antipodal positions on the cycle, as in our argument at the start of this section, and when a robot encounters a path attached to a vertex it can explore it and return to the cycle. In this manner, both robots visit every vertex of $G$ and by the above they always form a resolving set.
\end{proof}

For vertices $u,v$ of a graph $G$ we write $u \sim _N v$ if $u$ and $v$ are twins, that is, $N(u) \setminus \{ u\} = N(v) \setminus \{ v\}$. It is known that $\sim _N$ is an equivalence relation on $V(G)$ and each equivalence class $[u]_N$ is either an independent set or induces a clique~\cite{Hernando}. We call each equivalence class a \emph{twin class} and denote the number of twin classes by $\twin (G)$. Notice also that if $[u]_N \neq [v]_N$, then between $[u]_N$ and $[v]_N$ there are either no edges in $G$ or all possible edges are present.

\begin{lemma}\label{lem:twins}
  If $[u]_N$ is a twin class of $G$, then any mobile resolving set must always have at least $|[u]_N|-1$ robots on $[u]_N$. Hence $\mmd _c(G) \geq n-\twin (G)$.
\end{lemma}
\begin{proof}
    If there are two unoccupied vertices $v,v' \in [u]_N$, these have the same distance to each of the robots, and hence would not be resolved.
\end{proof}

\begin{theorem}\label{thm:complete multipartite}
    For $r \geq 2$, $n_1 \geq n_2 \geq \ldots \geq n_r \geq 1$, let $t$ be the number of parts of size $1$ and $s$ be the number of parts of size at least $ 2$ in the complete multipartite graph $K_{n_1,n_2,\ldots ,n_r}$. If $s$ and $t$ are not both $1$, then
\[
\mmd_c(K_{n_1,n_2,\ldots ,n_r}) =
\begin{cases}
n-s+1; & \text{ if } t = 0,\\
n-s; & \text{ if } s,t > 0,\\
n-1; & \text{ if } s = 0.
\end{cases}
\]
\end{theorem}
\begin{proof}
    Denote the parts of the partition by $V_1,V_2,\ldots ,V_r$, where $|V_i| = n_i$ for $1 \leq i \leq r$. The case $s = 0$ is just the complete graph $K_n$, so we can assume that $s > 0$. Let $S$ be any completely mobile resolving set of the graph. Let $W$ be the union of the parts of size $1$. By Lemma~\ref{lem:twins}, each of the $s$ parts of size at least $ 2$ has at most one unoccupied vertex, and if $t > 0$, then there is at most one unoccupied vertex in $W$. Thus, $|S| \geq n-s$ if $t = 0$; and $|S| \geq n-s-1$ if $t > 0$.

    First, suppose that $t = 0$. If there are $n-s$ robots in $S$, then any valid move would leave two unoccupied vertices in some part, contradicting Lemma~\ref{lem:twins}, so there must be $n-s+1$ robots. We now show that there is a completely mobile resolving set with $n-s+1$ robots. We start with one unoccupied vertex in each part, except for one part (call it the \emph{full part}), which has a robot at each vertex. It is easy to see that any such configuration is a resolving set. Now suppose that we wish to move a robot from a vertex $u \in V_i$ to a vertex $v \in V_j$. If there is not a robot at every vertex of $V_i$, move a robot from the full part to $V_i$. Assume firstly that $i \neq j$. Now if $v$ is unoccupied, we can make the valid move $u \move v$ straight away. Otherwise, if $v'$ is the unoccupied vertex of $V_j$ and $u' \in V_i\setminus \{ u\} $, make the series of moves $u' \move v'$, $v \move u'$ and $u \move v$. Now, if $i = j$, that is, we want to move a robot from $u$ in $V_i$ to another vertex $v$ of $V_i$, again make $V_i$ the full part. Let $V'$ be any other part of the partition and $v',v''$ be distinct vertices of $V'$, where $v'$ is unoccupied. Make the moves $u \move v'$, $v'' \move u$, $v \move v''$ and finally $v' \move v$, which places the robot at the desired vertex $v$.

    If $s,t > 0$ and there is an unoccupied vertex in each of the parts of size at least $ 2$ and an unoccupied vertex in $W$, then no robot from a part of size at least $2$ can make a valid move and the robots in $W$ cannot leave $W$. Hence, such a configuration is not completely mobile. However, if $s$ and $t$ are not both $1$ and we start with a robot at each vertex of $W$ and one unoccupied vertex in each part of size at least $2$, then an argument similar to the preceding case shows that this configuration is a completely mobile resolving set.
\end{proof}

\begin{corollary}
    The difference $\res (G)-\mmd _c(G)$ can be arbitrarily large for 2-connected graphs.
\end{corollary}
\begin{proof}
    It is shown in~\cite{resolving} that the resolving number of a complete multipartite graph with order $n$ is $n-1$. Therefore if $G$ is the complete multipartite graph with $k+2$ parts of size two, then by Theorem~\ref{thm:complete multipartite} we have $\res (G) -\mmd _c(G) = k$.
\end{proof}

We end this section by characterising graphs with order $n$ and completely mobile metric dimension $n-1$.

\begin{theorem}
  Let $G$ be a graph of order $n>1$. Then, $\mmd_c(G) = n-1$ if and only if $G$ is a join of two graphs which are both either a clique or a null graph.
\end{theorem}
\begin{proof}
 Firstly, by Theorem~\ref{thm:complete multipartite}, if $G = A \vee B$, where $A$ and $B$ are both either a clique or a null graph, then $\mmd_c(G) = n-1$.

 Now let $G$ be a graph of order $n$ which is not a join of two graphs which are both either a clique or a null graph (and so $\twin (G) \geq 3)$. Suppose for a contradiction that $G$ contains a cut vertex $v$ and that $\mmd_c(G) = n-1$. Then there must be a step when the set of occupied vertices is $V(G) \setminus \{ v \}$.  However, observe that in this configuration no robot is able to move to another component of $G-v$, and so, the robots are not completely mobile, a contradiction. Hence, we now assume that $G$ is 2-connected.

  We now show that $\mmd_c(G) \leq n-2$. Begin by stationing $n-2$ robots in $G$ so that the two unoccupied vertices lie in distinct twin classes. If the unoccupied vertices are $u,v$, then $u$ and $v$ are not twins, so there is an occupied vertex $w$ that is adjacent to just one of $u,v$, and hence resolves the pair. Therefore the configuration is a resolving set of $G$.

  Let $u, v \in V(G)$ be such that the robot $R$ is stationed at $u$. We show how to move $R$ to $v$ whilst always ensuring that the unoccupied vertices lie in distinct twin classes, completing the proof.

  First, suppose $[u]_N = [v]_N$ and $[u]_N$ induces a clique. If there is an unoccupied vertex $w \in [u]_N$, then move $v \move w$ (or do nothing if $v = w$) and then $u \move v$. Otherwise, consider a shortest path $v = v_0, v_1, \dots, v_k = w$ between $v$ and the set of two unoccupied vertices. Notice that such a path intersects any equivalence class at most once and for $0 \leq i < k$ there is no unoccupied vertex in $[v_i]_N$. By moving $v_{i} \move v_{i+1}$ for $i=k-1, \dots, 0$ in turn, $v$ becomes unoccupied, and the two unoccupied vertices are in distinct twin classes throughout. We may then move $u \move v$.

  Next, suppose that $[u]_N = [v]_N$ and $[u]_N$ is an independent set. First, if $[u]_N$ has an unoccupied vertex $w$, then we can move a robot along a shortest path from $w$ to a fully occupied twin class to ensure that $[u]_N$ is fully occupied. Suppose there exist $w_1, w_2 \in N(u) \setminus [u]_N$ with $[w_1]_N = [w_2]_N$. If possible, choose $w_1$ to be unoccupied. If not, since $G$ is 2-connected we may move the robots along any shortest path not through $v$ between $w_1$ and an unoccupied vertex so that $w_1$ becomes unoccupied. Then, move $u \move w_1$, $w_2 \move u$, $v \move w_2$, and $w_1 \move v$. Since $G$ is 2-connected and not a clique, if there are no such $w_1, w_2$, then there must exist $w_1, w_2 \in N(u) \setminus [u]_N$ with $|[w_1]_N| = |[w_2]_N| = 1$. By our choice of $w_1$ and $w_2$ and performing similar moves to before, we may ensure that $w_1$ is unoccupied and $w_2$ is occupied. Then the sequence $u \move w_1$, $w_2 \move u$, $v \move w_2$, $w_1 \move v$ moves $R$ to $v$.

  Finally, suppose that $[u]_N \neq [v]_N$. Let $w \in N(u)$ be such that $d(w, v) < d(u, v)$. First, perform moves not involving $R$ so that $[u]_N$ is fully occupied (if necessary). Then, (if necessary) perform moves not involving $R$ so that $[w]_N$ contains an unoccupied vertex, before moving $R$ to $[w]_N$. By following this procedure, we may move $R$ to $[v]_N$. The moves described earlier then allow us to move $R$ to $v$.
\end{proof}

\section{Characterising the non-mobile-resolvable graphs}\label{sec:exclud-non-res}

It was shown in~\cite{camacho} that every graph has a mobile resolving set. We also saw in Lemma~\ref{resolving bound} that all 2-connected graphs are mobile-resolvable. On the other hand, we have seen that path graphs of order $n \geq 3$ are not mobile-resolvable. This raises the question: which separable graphs have no completely mobile resolving set? We completely characterise these graphs in the following theorem.

\begin{theorem}
\label{th:charat-non-CMR}
Let $G$ be any connected graph. Then $G$ is non-mobile-resolvable if and only if it is a subdivided star graph $S_n(k)$ with $n \geq 3$ and $k \leq 2$, or a path of order at least three.
\end{theorem}

In order to prove this theorem, we first require some primary results and a few technical tools. Specifically, we first need to consider the class of trees, with special emphasis on subdivided star graphs, which are considered next.

\begin{proposition}
\label{prop:subd-not-res}
Let $k\ge 0$ and $n\ge 3$ be two integers. Then a subdivided star graph $S_n(k)$ is non-mobile-resolvable if and only if $k\le 2$.
\end{proposition}

\begin{proof}
Let $z$ be the unique major vertex of $S_n(k)$. Assume first $k=0$, that is the case of the star $S_n$. It is very well known (see Subsection \ref{subsec:basic}) that any metric basis of $S_n$ has exactly $n-1$ vertices and that it is formed by a set $S$ of any $n-1$ leaves of $S_n$. It can be readily observed that if we place $n-1$ robots in a metric basis of $S_n$, then no robot can make any resolving movement. Thus, a metric basis is not a completely mobile resolving set. Also, if we consider a set $S$ of $n$ vertices in $S_n$, then there are two options.
\begin{itemize}
    \item $S$ is formed by all $n$ leaves of $S_n$, in which case each robot placed at any of the leaves of the star can visit $z$, but cannot visit a leaf of $S_n$ other than that where it started.
    \item $S$ is formed by $n-1$ leaves of $S_n$ and the vertex $z$, in which case, a robot placed at $z$ can visit the free leaf (the one not in $S$) and any of the robots placed at leaves of the star can visit $z$, but cannot visit any leaves of $S_n$ other than that where they started.
\end{itemize}
As a consequence, $S$ is not a completely mobile resolving set. Since the whole set of vertices of $S_n$ cannot be a completely mobile resolving set, we deduce that $S_n$ is not a mobile-resolvable graph.

Assume now that $k\ge 3$. Also, let $P^1,P^2, \dots, P^k$ be the non-trivial terminal paths of $S_n(k)$, let $u_1,u_2,\dots,u_k$ be the leaves of $S_n(k)$, and $w_1,w_2,\dots,w_k$ be the neighbours of $z$, belonging to $P^1,P^2, \dots, P^k$, respectively. Moreover, if $n>k$, then let $v_1,\dots, v_{n-k}$ be the leaves of $S_n(k)$ adjacent to $z$. Consider a set of $n+1$ robots placed at $S=\{u_1,u_2,\dots,u_k\}\cup \{v_1,\dots, v_{n-k}\}\cup \{x\}$ where $x$ is the neighbour of $u_1$ (notice that $x\ne z$ and that it might happen that $x=w_1$); see Figure \ref{fig-tree-1} for an example.

\begin{figure}[ht]
\centering
\begin{tikzpicture}[scale=.7, transform shape]
\node [draw, shape=circle] (c) at  (0,0) {};
\node [draw, shape=circle, fill=black] (a1) at  (-4,2) {};
\node [draw, shape=circle, fill=black] (a2) at  (-4,-2) {};
\node [draw, shape=circle] (b1) at  (4,2) {};
\node [draw, shape=circle] (b2) at  (4,-2) {};
\node [draw, shape=circle, fill=black] (c1) at  (8,2) {};
\node [draw, shape=circle] (c2) at  (8,-2) {};
\node [draw, shape=circle, fill=black] (d1) at  (12,2) {};
\node [draw, shape=circle, fill=black] (d2) at  (12,-2) {};
\node [draw, shape=circle, fill=black] (e1) at  (-4,0) {};
\node [draw, shape=circle] (e2) at  (4,0) {};
\node [draw, shape=circle, fill=black] (e3) at  (8,0) {};

\node at (-4.5,2) {$v_1$};
\node at (-4.5,0) {$v_2$};
\node at (-4.5,-2) {$v_3$};
\node at (0,0.5) {$z$};
\node at (4,2.5) {$w_1$};
\node at (4,0.5) {$w_2$};
\node at (4,-1.5) {$w_3$};
\node at (12,2.5) {$u_1$};
\node at (8,0.5) {$u_2$};
\node at (12,-1.5) {$u_3$};

\node at (8,2.5) {$x$};

\draw(a1)--(c)--(b1)--(c1)--(d1);
\draw(e1)--(c)--(e2)--(e3);
\draw(a2)--(c)--(b2)--(c2)--(d2);
\end{tikzpicture}
\caption{The initial set $S$ in a subdivided star $S_6(3)$.}
\label{fig-tree-1}
\end{figure}
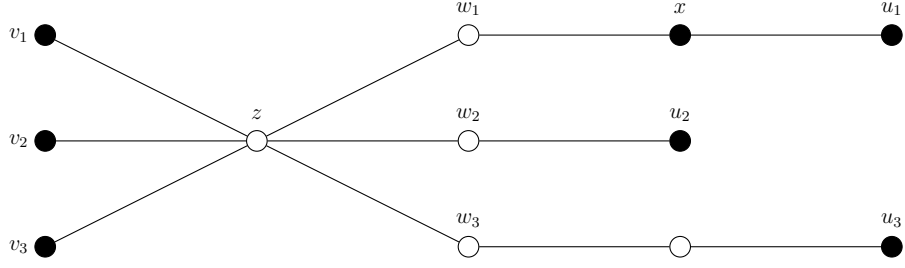

In order to see that such a set is a completely mobile resolving set, we first observe that the robot placed at the vertex $x$ can legally visit every vertex of the graph, except those where a robot is placed in the initial configuration. In this sense, in order to complete our arguments, we only need to show that any robot placed at a vertex from $S\setminus\{x\}$ and the robot placed at $x$ can always exchange positions. To do this, assume $R_x$ is the robot placed at $x$.

\smallskip
\noindent
{\em Situation 1}- Robot $R$ placed at $v_j$ and $R_x$. We perform the sequences $R_x: x\move\cdots\move z\move w_2$, $R: v_j\move z\move\cdots\move x$, and then $R_x: w_2\move z\move v_j$, which completes the required exchange.

\smallskip
\noindent
{\em Situation 2}- Robot $R$ placed at $u_1$ and $R_x$. We perform the sequences $R_x: x\move\cdots\move z\move w_2$, $R: u_1\move x\move\cdots\move z\move w_3$, next $R_x: w_2\move z\move\cdots\move x\move u_1$, and end with $R: w_3\move z\move\cdots\move x$, which again completes the required exchange.

\smallskip
\noindent
{\em Situation 3}- Robot $R$ placed at $u_j$ with $j\ne 1$ and $R_x$. We perform the sequences $R_x: x\move\cdots\move z\move w_3$, $R: u_j\move\cdots\move z\move\cdots\move x$, and then $R_x: w_3\move z\move\cdots\move u_j$, to make the required exchange.

\medskip
Notice that these are all the possible exchanges between any two robots from the set $S$ where one of the robots is stationed at $x$. Therefore, we deduce that each robot from the set $S$ can visit any vertex of the graph $S_n(k)$ when $k\ge 3$, and so, it is a mobile-resolvable graph.

\bigskip
Assume finally that $k\in \{1,2\}$. In such a case, there are either one or two non-trivial terminal paths in $S_n(k)$. Also, since $n\ge 3$, there must be at least one leaf adjacent to $z$, say $z'$. Moreover, as in the case $k=0$, a metric basis is not a completely mobile resolving set, which means that each completely mobile resolving set contains at least $n\ge 3$ vertices. Suppose now there is a completely mobile resolving set $S$ of $S_n(k)$ with at least $n$ vertices, and station a robot at each of these vertices.

Since each robot must visit every vertex of $S_n(k)$ and because there are at least $n\ge 3$ robots, in order that every robot will visit the vertex $z'$, there must be a robot $R$ which performs the move $z\move z'$. Immediately before this, the vertex $z'$ must be unoccupied. Moreover, before the previous move $x\move z$ of $R$, (where $x \sim z$), the leaf $z'$ was also not occupied by a robot. If there is no other robot in the terminal path in which $x$ lies, then the move $x\move z$ is not a valid move, because the two vertices $x,z'$ are not identified by the set of robots at this stage. Thus, $x$ is not a leaf, and it belongs to a non-trivial terminal path, say $P$, in which there is stationed at least one another robot, say $R'$.

Consequently, there are at least two robots $R$ and $R'$ respectively located at two vertices $x,y$ from a non-trivial path $P$. By our arguments, $x$ is closer to $z$ than $y$, and we may without loss of generality assume that $y$ is the closest robot to the terminal leaf of $z$ lying in $P$ (there might be other robots in $P$). Since all the robots need to visit all the vertices, $R'$ also has to visit $z'$. To do so, $R$ needs to legally move to $z$, and further on to a vertex from a terminal path other than $P$, say $P'$, if such a path exists; or to an unoccupied leaf adjacent to $z$ (which must be $z'$). In the first case, we immediately observe that $k=2$. In the second case, to visit $z'$ later, a move $z'\move z\move w$ for a vertex $w$ of $P'$ must occur, giving $k=2$.

After this, all the remaining robots that were occupying positions in the non-trivial terminal path $P$ between $x$ and $y$ need to make similar valid moves until they are all located in $P'$. Then, $R'$ can move to $x$, and next would need to legally move as $x\move z\move z'$ to visit $z'$. However, already the first move $x\move z$ is not valid, because after it, there is no robot in $P$, and then $x,z'$ are not resolved by the set of robots in this moment. This means that $R'$ cannot visit the vertex $z'$, which is a contradiction. Thus, $S_n(k)$ does not have a completely mobile resolving set.

We deduce that $S_n(k)$ is not a mobile-resolvable graph when $k\in \{1,2\}$, and the proof is completed.
\end{proof}

The result above is now complemented with a complete study on the completely mobile metric dimension of trees. To do this, we first need the following result.

\begin{remark}
\label{rem:mmdt-trees}
For any mobile-resolvable tree $T$, $\mmd_c(T)\ge \dim(T)+1$.
\end{remark}

\begin{proof}
The bound can be readily observed by the already mentioned known structure of any metric basis of a tree. That is, if we select a metric basis $S$ of a tree $T$ and place a robot in each of the vertices of $S$, then in order to visit one of the exterior major vertices of $T$, say $w$, there will be two terminal paths of $w$ in which no robot is stationed. Thus, there will be two neighbours of $w$ that will not be resolved by the set of vertices occupied by the robots, which is not possible. Thus, no metric basis of $T$ forms a completely mobile resolving set of $T$, and the result follows
\end{proof}

Notice that the bound above can be also obtained from \eqref{eq-completely-mobile-not-tot} and a result already presented in~\cite[Corollary 2.5]{camacho}, saying that for any tree $T$, it holds $\mmd(T)=\dim(T)+1$.

\medskip
Having the remark above in mind, we now consider the completely mobile metric dimension of trees.

\begin{theorem}
\label{thm:mmdt-trees}
For any tree $T$ different from a path or from a subdivided star $S_n(k)$ with $k\le 2$,
$$\mmd_c(T)=\left\{\begin{array}{ll}
    \dim(T)+2; & \mbox{if $T$ is a subdivided star $S_n(k)$ with $k\ge 3$}\\
    \dim(T)+1; & \mbox{otherwise.}
\end{array}\right.$$
\end{theorem}

\begin{proof}
Assume first $T$ is not a subdivided star $S_n(k)$. Consider a set $S$ of vertices of the tree $T$ containing all but one terminal leaf from each of the exterior major vertices of the tree, except from exactly one exterior major vertex, which has at least two terminal leaves, all of which we include in $S$. Notice that since $T$ is not a path, it has such an exterior major vertex. In addition, by the assumption that $T$ is not a subdivided star, $T$ has at least two exterior major vertices with at least two terminal leaves.  We station one robot at each of the vertices of $S$. See Figure \ref{fig-tree-proof} for a representative example.

\begin{figure}[ht]
\centering
\begin{tikzpicture}[scale=.57, transform shape]
\node [draw, shape=circle] (x1) at  (0,0) {};
\node [draw, shape=circle] (y1) at  (2,0) {};
\node [draw, shape=circle] (z1) at  (4,0) {};
\node [draw, shape=circle] (x2) at  (7,0) {};
\node [draw, shape=circle] (y2) at  (9,0) {};
\node [draw, shape=circle] (z2) at  (11,0) {};
\node [draw, shape=circle] (x3) at  (14,0) {};
\node [draw, shape=circle] (y3) at  (16,0) {};
\node [draw, shape=circle] (z3) at  (18,0) {};
\node [draw, shape=circle] (x4) at  (21,0) {};
\node [draw, shape=circle] (y4) at  (23,0) {};
\node [draw, shape=circle] (z4) at  (25,0) {};
\node [draw, shape=circle, fill=black] (y22) at  (9,-2.5) {};
\node [draw, shape=circle, fill=black] (x33) at  (14,-2.5) {};
\node [draw, shape=circle, fill=black] (z33) at  (18,-2.5) {};
\node [draw, shape=circle, fill=black] (x44) at  (21,-2.5) {};
\node [draw, shape=circle, fill=black] (y44) at  (23,-2.5) {};
\node [draw, shape=circle, fill=black] (z44) at  (25,-2.5) {};
\node [draw, shape=circle, fill=black] (x444) at  (0,-2.5) {};
\node [draw, shape=circle, fill=black] (y444) at  (4,-2.5) {};
\node [draw, shape=circle, fill=black] (z444) at  (7,-2.5) {};
\node [draw, shape=circle, fill=gray] (c1) at  (2,3) {};
\node [draw, shape=circle, fill=gray] (c2) at  (9,3) {};
\node [draw, shape=circle, fill=gray] (c3) at  (16,3) {};
\node [draw, shape=circle, fill=gray] (c4) at  (23,3) {};
\node [draw, shape=circle, fill=gray] (d1) at  (5.5,3) {};
\node [draw, shape=circle, fill=gray] (d2) at  (19.5,3) {};
\node [draw, shape=circle] (e1) at  (5.5,5) {};
\node [draw, shape=circle] (e2) at  (19.5,5) {};

\draw(x1)--(c1)--(y1);
\draw(c1)--(z1);
\draw(d1)--(e1);
\draw(d2)--(e2);

\draw(x444)--(x1);
\draw(y444)--(z1);
\draw(z444)--(x2);

\draw(c1)--(d1)--(c2)--(c3)--(d2)--(c4);
\draw(y22)--(y2)--(c2)--(x2);
\draw(c2)--(z2);

\draw(x33)--(x3)--(c3)--(z3)--(z33);
\draw(c3)--(y3);

\draw(x44)--(x4)--(c4)--(z4)--(z44);
\draw(c4)--(y4)--(y44);
\end{tikzpicture}
\caption{A tree $T$ provided with a set $S$ (in bold) as described in the proof. Exterior major vertices appear in gray.}
\label{fig-tree-proof}
\end{figure}
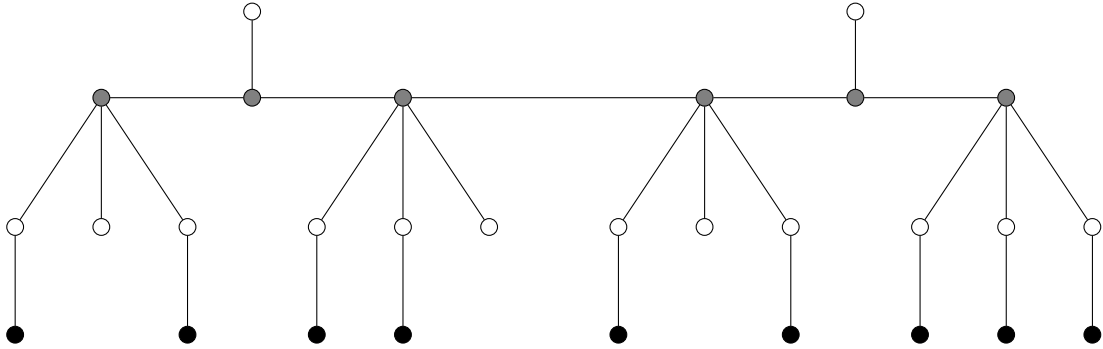

We now show that $S$ is a completely mobile resolving set of $T$. 
Since $T$ is not a subdivided star, we may consider two of its exterior major vertices, say $x$ and $y$ such that $\td(x)\ge 2$, $\td(y)\ge 2$, and assume w.l.o.g. that all the terminal leaves of $x$ are in $S$. Recall that, according to the structure of the metric basis of a tree, outlined in Section \ref{subsec:basic}, if an exterior major vertex $z$ satisfies that $\td(z)=1$, then $S$ has no vertex in the unique terminal path of $z$.

We first observe that any robot stationed in a terminal leaf of $x$, say $x_1$, can make a sequence of valid moves to visit each vertex in $(V(T)\setminus S )\cup \{x_1\}$ before returning to its initial position. Notice that this sequence includes visiting exterior major vertices whose terminal degree equals one together with the vertices of their terminal paths.

Now, let $x_2\ne x_1$ be another terminal leaf of $x$, and let $y_1,y_2$ be terminal leaves of $y$. Assume that initially $y_1\in S$ and $y_2 \notin S$.  Let $R_1,R_2,R_3$ be the robots stationed at $x_1,x_2,y_1$, respectively. We next see that these three robots can exchange their positions and visit the four vertices $x_1,x_2,y_1,y_2$. This can be made by the following sequences of valid resolving movements.
\begin{itemize}
    \item $R_1: x_1\move\cdots\move y_2$; $R_3: y_1\move\cdots\move x_1$; $R_2: x_2\move\cdots\move y_1$.
    \item $R_1: y_2\move\cdots\move x_2$; $R_3: x_1\move\cdots\move y_2$; $R_2: y_1\move\cdots\move x_1$.
    \item $R_1: x_2\move\cdots\move y_1$; $R_3: y_1\move\cdots\move x_2$; $R_2: x_1\move\cdots\move y_2$.
\end{itemize}
Notice that whenever a robot from the three above is located at the vertex $x_1$, and all the terminal leaves of $x$ are occupied, the robot at $x_1$ can make a sequence of valid resolving movements, and visit all the vertices of $T$ except those already occupied by a robot.

Now, one may observe that the robot placed at $x_1$ and other robots can be considered in a similar manner to deduce the existence of valid sequences of resolving movements that allow the robots to visit all the vertices of $T$. Therefore, $S$ is a completely mobile resolving set of $T$, and so, $\mmd_c(T)\le |S|=\dim(T)+1$. In addition, by using Remark \ref{rem:mmdt-trees}, the required equality holds for the case when $T$ is  not a subdivided star.

\medskip
Assume now $T$ is a subdivided star $S_n(k)$ with $k\ge 3$ and let $w$ be its unique exterior major vertex. By using the arguments of the proof of Proposition \ref{prop:subd-not-res}, we deduce that $\mmd_c(S_n(k))\le \dim(S_n(k))+2$. Also, from Remark~\ref{rem:mmdt-trees}, it holds that $\mmd_c(S_n(k))\ge \dim(S_n(k))+1$. Suppose that $\mmd_c(S_n(k))=\dim(S_n(k))+1=n$ and let $S$ be completely mobile metric basis.

Due to the structure of a resolving set in a tree, $S$ must have one vertex in $n-1$ terminal paths of $w$, and the remaining vertex of $S$ can be anywhere in the subdivided star $S_n(k)$. It can be readily observed that any robot placed at a vertex from $S$ can legally visit all the vertices of $S_n(k)$ except their leaves. However, by using similar arguments as in the case $k\in\{1,2\}$ from the proof of Proposition \ref{prop:subd-not-res}, such a configuration of robots does not allow any exchange of position between any two robots. Thus, a robot placed at a leaf $x$ in subdivided star $S_n(k)$, cannot visit any other leaf of $T$ other than $x$, which is a contradiction. Therefore,  $\mmd_c(S_n(k))>\dim(S_n(k))+1$, and thus the desired equality follows for this case, and the proof is completed.
\end{proof}

We are now in position to present the proof of the key result from the section.

\subsection{Proof of Theorem \ref{th:charat-non-CMR}}

\smallskip
For a separable graph $G$, we construct the collection $G^\times$ of edge-disjoint subgraphs of $G$ obtained as follows. To begin, add each block of $G$ to $G^\times$. Recall that a block in $G$ is a maximal 2-connected subgraph of $G$. If there exist $T_1, T_2 \in G^\times$ which are both trees with $V(T_1) \cap V(T_2) \neq \emptyset$, then remove $T_1$ and $T_2$ from $G^\times$ and add the subgraph of $G$ induced by $V(T_1) \cup V(T_2)$. Repeat this until there are no such trees. Note that at the very first stage, the only trees in $G^\times$ will be isomorphic to $K_2$. After the procedure is completed, any tree in $G^\times$ is called a \emph{tree component}. All other elements of $G^\times$ are \emph{cyclical components} (since any two vertices are contained in a common cycle). Notice that any cyclical component has at least $3$ vertices. In order to help the reader, in Figure~\ref{fig:G-prime} we have exemplified the described procedure.

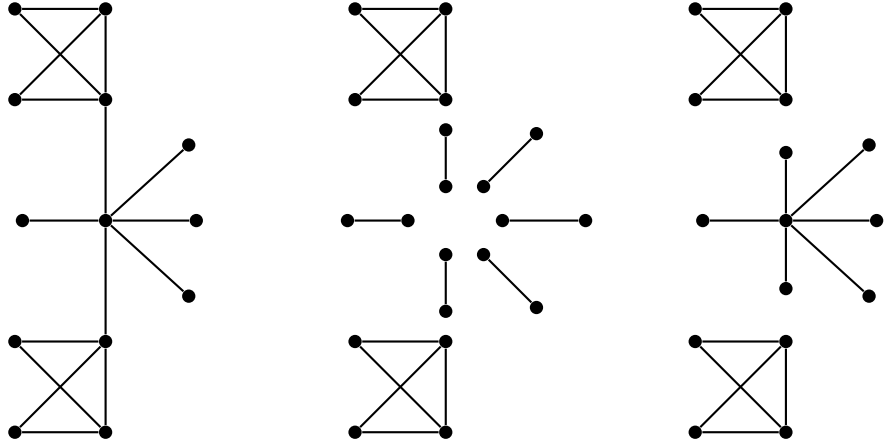
\begin{figure}[ht]
  \centering
\begin{tikzpicture}[
    v/.style={circle, fill=black, inner sep=0pt, minimum size=5pt},
    e/.style={thick}
]

\foreach \i/\xShift in {1/0, 2/4.5, 3/9} {

    \node[v] (t1_\i) at (\xShift + 0.0, 2.8) {};
    \node[v] (t2_\i) at (\xShift + 1.2, 2.8) {};
    \node[v] (t3_\i) at (\xShift + 0.0, 1.6) {};
    \node[v] (t4_\i) at (\xShift + 1.2, 1.6) {};

    \draw[e] (t1_\i) -- (t2_\i) -- (t4_\i) -- (t3_\i) -- cycle; 
    \draw[e] (t1_\i) -- (t4_\i); 
    \draw[e] (t2_\i) -- (t3_\i); 

    \node[v] (b1_\i) at (\xShift + 0.0, -1.6) {};
    \node[v] (b2_\i) at (\xShift + 1.2, -1.6) {};
    \node[v] (b3_\i) at (\xShift + 0.0, -2.8) {};
    \node[v] (b4_\i) at (\xShift + 1.2, -2.8) {};

    \draw[e] (b1_\i) -- (b2_\i) -- (b4_\i) -- (b3_\i) -- cycle; 
    \draw[e] (b1_\i) -- (b4_\i); 
    \draw[e] (b2_\i) -- (b3_\i); 
}

\node[v] (c1) at (1.2, 0) {};
\node[v] (l1) at (0.1, 0) {};
\node[v] (r1_1) at (2.3, 1.0) {};
\node[v] (r1_2) at (2.4, 0.0) {};
\node[v] (r1_3) at (2.3, -1.0) {};

\draw[e] (t4_1) -- (c1) -- (b2_1);
\draw[e] (l1) -- (c1);
\draw[e] (c1) -- (r1_1);
\draw[e] (c1) -- (r1_2);
\draw[e] (c1) -- (r1_3);

\node[v] (v_top_1) at (5.7,  1.2) {};
\node[v] (v_top_2) at (5.7,  0.45) {};
\draw[e] (v_top_1) -- (v_top_2);

\node[v] (v_bot_1) at (5.7, -0.45) {};
\node[v] (v_bot_2) at (5.7, -1.2) {};
\draw[e] (v_bot_1) -- (v_bot_2);

\node[v] (h_left_1) at (4.4, 0.0) {};
\node[v] (h_left_2) at (5.2, 0.0) {};
\draw[e] (h_left_1) -- (h_left_2);

\node[v] (d_tr_1) at (6.2,  0.45) {};
\node[v] (d_tr_2) at (6.9,  1.15) {};
\draw[e] (d_tr_1) -- (d_tr_2);

\node[v] (h_right_1) at (6.45, 0.0) {};
\node[v] (h_right_2) at (7.55, 0.0) {};
\draw[e] (h_right_1) -- (h_right_2);

\node[v] (d_br_1) at (6.2, -0.45) {};
\node[v] (d_br_2) at (6.9, -1.15) {};
\draw[e] (d_br_1) -- (d_br_2);

\node[v] (c3) at (10.2, 0) {};
\node[v] (t3_ext) at (10.2, 0.9) {};
\node[v] (b3_ext) at (10.2, -0.9) {};
\node[v] (l3) at (9.1, 0) {};
\node[v] (r3_1) at (11.3, 1.0) {};
\node[v] (r3_2) at (11.4, 0.0) {};
\node[v] (r3_3) at (11.3, -1.0) {};

\draw[e] (t3_ext) -- (c3) -- (b3_ext);
\draw[e] (l3) -- (c3);
\draw[e] (c3) -- (r3_1);
\draw[e] (c3) -- (r3_2);
\draw[e] (c3) -- (r3_3);

\end{tikzpicture}
\caption{An example of generating the collection $G^\times$. In the left hand side, a graph $G$ is drawn. In the first step, all the subgraphs from the center are included in $G^\times$. In the last step, the final collection of subgraphs in $G^\times$ drawn in the right hand side is obtained.}\label{fig:G-prime}
\end{figure}

Let $C$ be the set of cut vertices of $G$ which are contained in a cycle.
Distinct components $H_1, H_2 \in G^\times$ are \emph{attached $($by} $v$), if $V(H_1) \cap V(H_2) = \{ v \}$ (by our construction, $v \in C$). Note that, by the construction procedure, no two tree components are attached (by a vertex) and each vertex in $V(G) \setminus C$ lies in exactly one component in $G^\times$. The following fact can be then observed on the basis of our construction, where $G-v$ represents the graph obtained from $G$ by removing the vertex $v\in V(G)$ and all the edges incident with it.

\begin{fact}
  \label{fac:cutreplacement}
Let $S_H \subseteq V(G)$ be a set that resolves the vertices of a component $H \in G^\times$. If $v \in S_H \cap V(H)$ and $w \notin S$ is a vertex such that $d(u, w) = d(u, v) + d(v, w)$ for all $u \in H$, then $S \setminus \{ v \} \cup \{ w \}$ also resolves the vertices of $H$.
\end{fact}

Now, for any $S \subseteq V(G)$ and $H \in G^\times$, we shall say that $v \in V(H)$ is \emph{quasi-occupied in $H$ $($under $S$$)$} if there exists $v^o \in S$ such that $d(u, v^o) = d(u, v) + d(v, v^o)$ for all $u \in V(H)$. Notice that, in this sense, if $v \notin C$, then $v$ is quasi-occupied in $H$ if and only if $v \in S$.

\begin{lemma}
  \label{lem:quasiresolving}
  Let $S \subseteq V(G)$ be such that for all tree components $T \in G^\times$, the set of quasi-occupied vertices in $T$ resolves $T$ and for all cyclical components $H \in G^\times$, at most one vertex in $H$ is not quasi-occupied. Then, $S$ is a resolving set of $G$.
\end{lemma}

\begin{proof}
  Let $u,v \in V(G) \setminus S$ be distinct vertices and $H_u, H_v \in G^\times$ be such that $u \in V(H_u)$ and $v \in V(H_v)$. We should choose $H_u = H_v$ if possible. If $H_u = H_v$, then $u$ and $v$ are resolved by Fact~\ref{fac:cutreplacement}. We consider the two other cases in turn.

  First suppose that $H_u \neq H_v$ and $H_u$ is a cyclical component. By our assumption and the fact that cyclical components have at least $3$ vertices, it holds that at least two vertices in $N_{H_u}[u]$ are quasi-occupied. 
  We see that $u$ and $v$ are resolved by one of the corresponding occupied vertices.

  Now assume that $H_u$ and $H_v$ are distinct tree components. Let $c_u$ and $c_v$ be the first and last vertices in $C$ in any shortest $u,v$-path. Both $c_u$ and $c_v$ lie in cyclical components $J_u$ and $J_v$, respectively. First, consider that we may choose $J_u$ and $J_v$ so that $N_{J_u}(c_u) \nsubseteq N_{G}[c_v]$ and $N_{J_v}(c_v) \nsubseteq N_{G}[c_u]$. Then there exists $w \in N_{J_u}[c_u]$ which is quasi-occupied with $d(v, w^o) = d(v, c_v) + d(c_v, w) + d(w, w^o) > d(v, c_v) + d(c_u, w) + d(w, w^o)$ and $d(u, w^o) \leq d(u, c_u) + d(c_u, w) + d(w, w_o)$. Either $w^o$ resolves $u$ and $v$, or $d(v, c_v) < d(u, c_u)$. But by the same reasoning, with $u$ and $v$ interchanged, if $u$ and $v$ are not resolved, then we obtain $d(v, c_v) > d(u, c_u)$, which is a contradiction.

  Next, assume that there is no such choice of $J_u$ and $J_v$ and take $J_u = J_v$.
  Then, either $c_u \sim c_v$ and both vertices are quasi-occupied in $J_u$, or there exists $w \in N_{J_u}(c_u) \cap N_{J_u}(c_v)$ which is quasi-occupied. In the first case, either $c_u^o$ resolves $u$ and $v$ or $d(u, c_u) > d(v, c_v)$, so that $c_v^o$ resolves $c_u$ and $c_v$.
  In the second case, either $w$ resolves $u$ and $v$ or $d(u, c_u) = d(v, c_v)$. If $d(u, c_u) = d(v, c_v)$, since either $c_u$ or $c_v$ is quasi-occupied in $J_u$, either $c_u^o$ or $c_v^o$ resolves $u$ and $v$.
\end{proof}

In order to complete the proof of our Theorem \ref{th:charat-non-CMR}, from Proposition~\ref{prop:subd-not-res} and the discussion in the introduction, we only need to prove the `only if' direction. Hence, we now assume that $G$ is neither a subdivided star graph $S_n(k)$ with $n \geq 3$ and $k \leq 2$, nor a path of order at least three. We show that there exists a completely mobile resolving set of $G$. By Theorem~\ref{thm:mmdt-trees}, it suffices to assume that $G$ is not a tree.

\medskip
We construct a set $S$ of the form of Lemma~\ref{lem:quasiresolving}. First, for each tree component $T \in G^\times$ which is not an $S_n(k)$, add all vertices in a completely mobile resolving set of $T$ to $S$.
Add a single vertex to $S$ from each component isomorphic to a $P_n$ with at least one leaf vertex in $C$. In what follows, these components shall be regarded as paths (not subdivided stars). Any component isomorphic to a $P_n$ with no leaf vertex in $C$ should be regarded as an $S_2(k)$ with major vertex equal to one of its vertices in $C$.
For each component $T$ which is an $S_n(k)$ with $n \geq 2$, if $V(T) \cap C$ consists of the major vertex, add each leaf vertex to $S$. Otherwise, add $n-1$ leaf vertices to $S$.
Add any vertex in $C$ not lying in a tree component to $S$. For each cyclical component $H$, add all but one of the vertices in $V(H) \setminus S$ which do not attach $H$ to a tree component (if any such vertices exist).

We now show that $S$ is a completely mobile resolving set. In what follows, moves between vertices should be performed over a shortest path (unless otherwise specified). Begin by stationing robots at the vertices in $S$. Consider $H \in G^\times$ as a cyclical component. Let $u, v \in V(H)$ be such that $v$ is occupied by a robot $R$. We show how to move $R$ to $u$. For each unoccupied vertex which attaches $H$ to a tree component $T$, perform moves in $T$ resulting in the vertex becoming occupied.

If there exists $w \in V(H) \setminus (\{ v \} \cup (\{ u \} \cap C))$ which is either unoccupied or attaches $H$ to a tree component, then perform moves in the attached tree component, if necessary, so that $w$ becomes unoccupied. There is a $u,w$-path in $H$ of the form $u=v_0, v_1, \dots, v_r=w$ with $v_i \neq v$ occupied for $0 \leq i <r$. We move $v_i \move v_{i+1}$ for $i=r-1, r-2, \dots, 0$ in turn, so that $u$ is unoccupied. In a similar way, we then `shift' the robots in a $v,u$-cycle in $H$ as many times as necessary to move $R$ to $u$.

If we cannot choose such a $w$, then each vertex in $V(H) \setminus \{ v, u \}$ attaches $H$ to a cyclical component. Choose one such vertex $w$ and take the largest possible sequence of distinct components $H=H_0, H_1, \dots, H_r$ such that $w \in V(H_1)$ and $H_i$ and $H_{i+1}$ are attached by $c_i$ for $0 \leq i < r-1$. Either there exists a smallest $i$ such that $H_i$ is a tree component, or $H_r$ is a cyclical component with an unoccupied vertex. Either way, we may perform legal moves in $H_i$ so that $c_{i-1}$ becomes unoccupied. We may then perform legal moves in $H_{i-1}$ resulting in $c_{i-2}$ being unoccupied. Repeating this process, $w$ becomes unoccupied. We then proceed as before.

Now let $T \in G^\times$ be a tree component. It is clear from our construction that if $T$ is not an $S_n(k)$ with $n \geq 2$, then the robots in $T$ can traverse the component. Otherwise, since $G$ is not a tree, there must be a cyclical component $H$ attached to $T$ at some $u \in V(T)$. Avoid taking $u$ as the major vertex if possible. Label the terminal paths of $T$ as $\ell_1, \ell_2, \dots, \ell_n$, so that, if $u$ is not the major vertex, then it lies in $\ell_n$.
If there is a robot stationed at the major vertex, then move this into an unoccupied terminal path. Clearly, any robot in $T$ may visit all vertices in its terminal path.

We now show how to move a robot $R_1$ in the leaf $\ell_i$ to $\ell_j$. Without changing the occupancy of $u$, we may ensure that a vertex in $N_H(u)$ is unoccupied, but that all other vertices in $V(H) \setminus \{ u \}$ are occupied. If $u$ is not the major vertex, then move any robot in $\ell_n$ to the unoccupied leaf and then $R_1$ to $u$. If $j=n$ and $u$ lies in $\ell_n$, then we are done. Otherwise, shift the vertices in $H$ so that a different robot $R_3$ occupies $u$. Move $R_3$ into $\ell_i$. Then, move $R_2$ stationed in $\ell_j$ to $u$. Shift the vertices in $H$ so that $R_1$ is stationed at $u$. Then, move $R_1$ to $\ell_j$.

Now let $u$ and $v$ be vertices lying in different components of $G^\times$ with a robot $R$ stationed at $v$. We show how to move $R$ to $u$. Let $v=v_0, v_1, \dots, v_r = u$ be a shortest path in $G$. Let $v_{i_1}, v_{i_2}, \dots, v_{i_s}$ be the subsequence of $v_0, \dots, v_{r-1}$ in $C$. For $1 \leq j < s$, $v_{i_j}$ and $v_{i_{j+1}}$ both lie in a some $H_j \in G^\times$. We have already seen how to legally move $R$ from $v_{i_j}$ to $v_{i_{j+1}}$. However, if $H_{j}$ is a subdivided star and $H_{j+2}$ is a cyclical component for some $j < s-1$, then we should ensure that there is no move in $V(H_{j+2}) \setminus \{ v_{i_{j+2}} \}$ before $R$ reaches $v_{i_{j+1}}$. In this case, we may also need to perform moves in $V(H_{j+1}) \setminus \{ v_{i_{j+1}} \}$ so that $v_{i_{j+2}}$ is occupied after $R$ reaches $v_{i_{j+1}}$. If instead $H_{j+2}$ is a tree component, then once $R$ reaches $v_{i_{j+1}}$, we perform moves in $V(H_{j+1}) \setminus \{ v_{i_{j+1}} \}$ so that $v_{i_{j+2}}$ is occupied if and only if it was occupied after the last move in $V(H_{j+2}) \setminus \{ v_{i_{j+2}} \}$ (or at the beginning if there has been no such move). On reaching $v_{i_s}$, $R$ is in the same component as $u$ and we may perform legal moves previously described to reach $u$. At all times, the set of occupied vertices is of the form  of Lemma~\ref{lem:quasiresolving}. Hence, $G$ is completely mobile-resolvable, and this completes the proof of Theorem~\ref{th:charat-non-CMR}.

\section{Circulant graphs $C_n(2)$}\label{sec:cir-Cn-2}\label{sec:circulant}

The \textit{circulant graph} $C_n(2)$ is obtained from a cycle $C_n$ ($n\ge 4)$ by adding an edge between any two vertices at distance $2$ in $C_n$. If $4\le n\le 7$, then some calculations by hand or a few simple arguments allow us to obtain the following values: $\mmd_c(C_4(2))=3$, $\mmd_c(C_5(2))=4$ (which hold because $C_4(2)\cong K_4$ and $C_5(2)\cong K_5$); $\mmd_c(C_6(2))=4$; and $\mmd_c(C_7(2))=3$. The remaining larger values of $n$ are considered at next.

\begin{proposition}
\label{prop:Cn2}
If $G = C_n(2)$ is a circulant graph of order $n \geq 8$, then
$$
\mmd_c(G) =
\begin{cases}
4; & \text{if } n \equiv 1 \pmod{4},\\
3; & \text{otherwise}.
\end{cases}
$$
\end{proposition}

\begin{proof}
The proof is based on the existence of sequences of legal resolving movements where every intermediate state is a resolving set. Furthermore, we show that by repeating these sequences, every robot  eventually occupies every vertex $v \in V(G)$. Operations with the labels of the vertices are done modulo $n$.

\medskip
\noindent
\textbf{Case 1:} If $n \equiv 1 \pmod{4}$, then we begin with the robots placed at the set $\{i, i+1, i+2, i+3\}$ for some $i\in \{0,\dots, n-1\}$. The first sequence of moves is defined as:
 $$  \{i, i+1, i+2, i+3\} \move \{i, i+1, i+3, i+4\}  \move \{i+1, i+2, i+3, i+4\}.$$
These configurations are valid, as demonstrated in the article~\cite{camacho}. Note also that, since $n$ is odd, with a large enough number of repetitions of such moves, the jumps in the second step allow us to claim that the robots will eventually reach any vertex of the graph, regardless of the parity of its initial position.

\medskip
\noindent
\textbf{Case 2:} If $n \equiv 2,3 \pmod{4}$, then we begin with the robots placed at the set $S = \{0, 1, 2\}$, which is known to be a resolving set from~\cite{Javaid}. We focus on the move $S\move S' = \{0, 1, 3\}$.

The distance formula in $C_n(2)$ between any two vertices $i,j\in V(C_n(2))$ is given by
\begin{equation}
    \label{eq-dist}
d(i, j) = \min \left( \left\lceil \frac{|i-j|}{2} \right\rceil, \left\lceil \frac{n-|i-j|}{2} \right\rceil \right)
\end{equation}
We analyse the resolvability of $S' = \{0, 1, 3\}$ as follows.

Let us consider the set $\{0, 1\}$ and study the vertices that have equal distance to both. That is, assume there are $u,v\in V(C_n(2))$ such that $r(u|\{0,1\})=r(v|\{0,1\})$. This can only happen if $r(u|\{0,1\}) = (\lceil j/2 \rceil, \lceil (j-1)/2 \rceil)$ and $r(v,|\{0,1\}) = (\lceil (j-1)/2 \rceil, \lceil j/2 \rceil)$ for some even $j \in \{2, \dots, \lfloor n/2 \rfloor\}$. This in turns means that $u = j$ and $v = n+1-j$ are those pairs not resolved by the vertices $\{0, 1\}$. Consider now the vertex $3\in S'$.

\begin{itemize}
    \item For $u = j$: $d(u, 3) = \lceil \frac{j-3}{2} \rceil < \lceil \frac{j}{2} \rceil$.
    \item For $v = n+1-j$: $d(v, 3) = \lceil \frac{j+2}{2} \rceil = \lceil \frac{j}{2} \rceil + 1$.
\end{itemize}

So the third robot resolves the pairs $u,v$, ensuring that $S'$ is a resolving set. This allows to show that the sequence
\[ \{0, 1, 2\} \move \{0, 1, 3\} \move \{0, 2, 3\} \move \{1, 2, 3\} \]
is formed by legal resolving moves, where the set $\{0, 2, 3\}$ is a resolving set by symmetry with $S'$. By making analogous moves, the repetition of this process, leads to confirm that every robot can reach every vertex $v \in V(G)$, confirming that in this situation $\mmd_c(G) = 3$.

\medskip
\noindent
\textbf{Case 3:} If $n\equiv 0 \pmod{4}$, then we set $n=4k$ and begin with the robots placed at a set $S=\{0, 2k,j\}$ for some $j\in \{1,\dots,2k-1\}$. We begin by considering the set $B_0 = \{0, 2k\}$. Since $0$ and $2k$ are diametral, and based on the distance formula \eqref{eq-dist}, the pairs of vertices not resolved by $B_0$ have the form $i, n-i$ with $i \in \{1, \dots, 2k-1\}$. That is,
        \[ r(i|B_0) =  \left( \left\lceil \frac{i}{2} \right\rceil, \left\lceil \frac{2k-i}{2} \right\rceil \right), \]
    \[  r(n-i|B_0) = \left( \left\lceil \frac{i}{2} \right\rceil, \left\lceil \frac{4k-i-2k}{2} \right\rceil \right)=r(i|B_0) . \]
Consider next the third robot at a position $j \in \{1, \dots, 2k-1\}$. For a pair $\{i, n-i\}$ we have
    \[ d(i, j) = \min \left( \left\lceil \frac{|i - j|}{2} \right\rceil, \left\lceil \frac{n - |i - j|}{2} \right\rceil \right) \]
    \[ d(n-i, j) = \min \left( \left\lceil \frac{n - (i + j)}{2} \right\rceil, \left\lceil \frac{i + j}{2} \right\rceil \right) \]

We claim that these distances are distinct. To see this, consider the case $d(i, j) = \frac{|i-j|}{2}$ and $d(n-i, j) = \frac{n-(i+j)}{2}$. For these to be equal, the following must hold:
    \[ \frac{|i-j|}{2} = \frac{n-(i+j)}{2} \]
    Assuming $i > j$ for simplicity and multiplying by 2:
    \[ i - j = n - i - j \implies 2i = n \implies i = \frac{n}{2} = 2k \]
Since $i \notin \{0, 2k\}$, this is not possible. Thus, the values $d(i, j)$ and $d(n-i, j)$ are distinct. Similar conclusions are deduced if $d(i, j)$ and  $d(n-i, j)$ attain other values. This ensures that $B = \{0, 2k, j\}$ is a resolving set for the entire graph.

On the other hand, by the symmetry of the graph $G$, a set $B' = \{0, 2k, j'\}$ for any $j'\in \{2k+1,\dots,n-1\}$ is also a resolving set for $G$. Thus, the following facts are shown.

\medskip
\noindent
\textbf{Fact 1:} The set $B = \{0, 2k, j\}$ for any $j\in \{1,\dots,2k-1\}$ is a resolving set for $G$.

\medskip
\noindent
\textbf{Fact 2:} The set $B' = \{0, 2k, j'\}$ for any $j'\in \{2k+1,\dots,n-1\}$ is a resolving set for $G$.

\medskip
Now, let us begin with the initial configuration of the robots in $S = \{0, 2k, j\}$ for some $j\in \{1,\dots,2k-1\}$. The robots can move along all positions in $V(G)$ through the following sequence:
\begin{itemize}
    \item By the Fact 1, the robot placed at $j$ can legally move along all the vertices $i\in \{1,\dots,2k-1\}$ and finally remain at the vertex $1$.
    \item With the robots placed at $S'=\{0,1,2k\}$, the robot placed at $2k$ can move to the vertex $2k+1$, which is a legal move since the set $\{0,1,2k+1\}$ is a resolving set by Fact 2 and by the symmetry of $G$.
    \item The robot placed at $0$ can legally move (by using Fact 1) along all the vertices in $\{2k+2,\dots,n-1,0\}$ and finally remain at the vertex $2k+2$. Hence, the set set occupied by the robots is now $S''=\{1,2k+1,2k+2\}$.
    \item The robot placed at the vertex $1$ can legally move to the vertex $2$, which is based on Fact 2.
\end{itemize}
This process can be repeated, until all the vertices of $G$ are visited by all robots. This allows to confirm that $\mmd_c(G) = 3$ in this final situation, which completes the proof.
\end{proof}

\section{Open questions}\label{sec:open}

This paper has introduced a new model for network traversal by robots which always occupy a resolving set. We now conclude by suggesting several possible directions for further research.
\begin{itemize}
  \item A problem we have not addressed is that of finding the smallest number of moves required to traverse every vertex with every robot for a given completely mobile resolving set. To this end, it would be useful to develop algorithms to optimise the traversal with respect to the number of moves taken.
  \item Following Section \ref{sec:cir-Cn-2}, we propose a study of the completely mobile metric dimension of circulant graphs in general. Other families of graphs that could be studied include toruses, cylinders, Hamming graphs, and hypercubes.
  \item In this work, we have seen examples of graphs $G$ for which $\mmd(G)=\mmd_c(G)$. This suggests trying to find a characterisation of the class of graphs which satisfy this equality.
  \item It would also be interesting to find a characterisation of the class of graphs with completely mobile metric dimension $2$.
  \item Completely mobile resolving sets are a stronger version of the mobile resolving sets studied in~\cite{camacho}. This variant might be too strong for a practical implementation due to the potentially large number of moves required for a full traversal. Therefore, a future study could consider restricting the `range' that particular robots are required to traverse, or placing a limit on the number of vertices they are required to visit.
  \item As in many graph theoretical investigations, it is of interest to consider the complexity of computing the completely mobile metric dimension of graphs. For instance, is this problem in NP? Note that we may determine whether a given graph $G$ has $\mmd_c(G) = k$ in polynomial time.
\end{itemize}

\section*{Acknowledgements}

Dorota Kuziak and Ismael G. Yero have been partially supported by the Spanish Ministry of Science and Innovation through the grant PID2023-146643NB-I00.

\section*{Data availability}

No specific data was used in this investigation.

\section*{Conflict of interest disclosure statement}

We have no conflict of interest to declare.


\end{document}